\documentclass[final]{dmtcs-episciences}
\usepackage[utf8]{inputenc}
\usepackage{subfigure}
\usepackage{latexsym}
\usepackage{amsfonts}
\usepackage{amsmath}
\usepackage{amssymb}
\usepackage{amsthm}
\usepackage{enumerate}
\usepackage{caption}
\usepackage{tabu}
\usepackage{mathrsfs}

\newtheorem{theorem}{Theorem}[section]
\newtheorem{lemma}{Lemma}[section]
\newtheorem{proposition}{Proposition}[section]

\newtheorem{corollary}{Corollary}[section]

\theoremstyle{definition}
\newtheorem{definition}{Definition}[section]
\newtheorem{remark}{Remark}[section]
\newtheorem{example}{Example}[section]
\newtheorem{question*}{Question}

\author{Colin Defant\affiliationmark{1}}
\title[Postorder Preimages]{Postorder Preimages}
\affiliation{University of Florida}
\keywords{Tree traversal; decreasing plane tree; postorder; permutation; stack-sorting}
\received{2016-4-7}
\accepted{2017-1-31}
\revised{2017-1-19}
\begin{document}
\publicationdetails{19}{2017}{1}{3}{1428}
\maketitle

\begin{abstract}
Given a set $Y$ of decreasing plane trees and a permutation $\pi$, how many trees in $Y$ have $\pi$ as their postorder? Using combinatorial and geometric constructions, we provide a method for answering this question for certain sets $Y$ and all permutations $\pi$. We then provide applications of our results to the study of the deterministic stack-sorting algorithm. 
\end{abstract}

\section{Introduction}
If $X$ is a finite set of positive integers, then a \emph{decreasing plane tree on $X$} is a rooted plane tree with vertex set $X$ in which each nonroot vertex is smaller than its parent. Decreasing plane trees play a significant role in computer science. In that field, one is often interested in listing the vertices of the tree in some meaningful order. A scheme by which one reads these vertices is a \emph{tree traversal}. Two of the most common tree traversals are known as \emph{postorder} and \emph{preorder}; both are defined recursively. 

To read a decreasing binary plane tree (``binary" means that each vertex has at most two children) in postorder, we first read the left subtree of the root in postorder. We then proceed to read the right subtree in postorder before finally reading the root. This postorder traversal easily generalizes to any decreasing plane tree. Namely, we read the subtrees of the root in postorder from left to right before finally reading the root. As an example, the postorder of each tree in Figure \ref{Fig1} is $127358$. To read a tree in preorder, we read the root first, and then proceed to read the subtrees of the root from left to right in preorder. Alternatively, one can find the preorder reading of a tree by first reflecting the tree through a vertical axis and then taking the reverse of the postorder reading of the resulting tree. Because of this simple connection between the preorder and postorder, we will concern ourselves primarily with the postorder traversal; analogous results for the preorder traversal will follow trivially. 

Tree traversals have been incredibly useful tools in combinatorics and computer science. For example, the postorder reading has been instrumental in the study of the deterministic stack-sorting algorithm \cite{Albert14, Bona, Bona02, Bona03, Bousquet00, West90}. In fact, the study of this stack-sorting algorithm was the original motivation for developing the results in this article. However, there has been surprisingly little research concerning the combinatorics of the tree traversals themselves. In this article, we investigate the following very natural question. 

\begin{question*}\label{Quest1}
If $Y$ is a set of decreasing plane trees and $\pi$ is a permutation, then how many trees in $Y$ have postorder $\pi$? 
\end{question*}

Before proceeding, let us establish some terminology. We consider two major types of unlabeled plane trees. The first, which we call a $d$-ary plane tree, is either an empty tree or a root along with a $d$-tuple of $d$-ary plane trees. Therefore, the ordinary generating function $A_d(x)$ for $d$-ary plane trees satisfies the equation \[A_d(x)=1+xA_d(x)^d.\] The second type depends on a set $S$ of nonnegative integers with $0\in S$. Each of these trees, which we call $S$-trees, consists of a root along with a $j$-tuple of $S$-trees for some $j\in S$. When $S=\{0,1,2\}$, $S$-trees are commonly known as unary-binary trees. The ordinary generating function $B_S(x)$ for $S$-trees satisfies the functional equation \[B_S(x)=x\sum_{j\in S}B_S(x)^j.\] 

If $X$ is a finite set of positive integers, then a \emph{decreasing $d$-ary plane tree on} $X$ is a $d$-ary plane tree whose vertices are labeled with the elements of $X$ so that the label of any nonroot vertex is smaller than the label of its parent (distinct vertices are given distinct labels, and all elements of $X$ are used as labels). Similarly, a \emph{decreasing $S$-tree on} $X$ is an $S$-tree whose vertices are labeled with the elements of $X$ so that the label of any nonroot vertex is smaller than the label of its parent. In both types of trees, we will speak of the $h^\text{th}$ child of a vertex. This is simply the $h^\text{th}$ child from the left.  

\begin{example}\label{Exam1}
In this example let $S=\{0,1,2,3,4\}$. Figure \ref{Fig1} shows two decreasing plane trees. As decreasing $S$-trees, they are identical. In both trees, the first child of $7$ is $1$ and the second child of $7$ is $2$. However, if we instead view the trees as decreasing $4$-ary trees, then they are distinct. In this case, in the tree on the right, the first child of $7$ is $1$, the fourth child of $7$ is $2$, and the second and third children of $7$ are empty. This contrasts the situation in the tree on the left, in which the first and second children of $7$ are $1$ and $2$ while the third and fourth children of $7$ are empty. 

\begin{figure}[t]
\begin{center} 
\includegraphics[height=2.5cm]{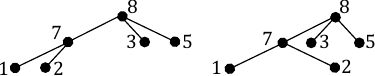}
\end{center}
\captionof{figure}{This figure depicts two decreasing plane trees on $X=\{1,2,3,5,7,8\}$. They are the same as decreasing $\{0,1,2,3,4\}$-trees, but they are different as decreasing $4$-ary trees.} \label{Fig1}
\end{figure}
\end{example}

We will give a bijective method for attacking Question \ref{Quest1} when $Y$ is the collection of decreasing $\mathbb N$-trees. We then show how our method applies to general decreasing $S$-trees and decreasing $d$-ary plane trees with only slight modifications. In the end, we apply a special case of our results to gain new information about the deterministic stack-sorting algorithm. 

A \emph{descent} of a permutation $\pi=\pi_1\pi_2\cdots\pi_n$ is an index $i$ such that $\pi_i>\pi_{i+1}$ (we do not consider $n$ to be a descent). A \emph{descent top} is simply an entry in a descent position. That is, a descent top is an entry $\pi_i$ such that $\pi_i>\pi_{i+1}$. We denote by $P(\tau)$ the postorder reading of a decreasing plane tree $\tau$. We convene to let $\mathbb N=\{0,1,2,\ldots\}$ and $[m]=\{1,2,\ldots,m\}$.

\section{Valid Hook Configurations} 

Our method relies on a geometric construction of objects that we call \emph{valid hook configurations}. The purpose of this section is to describe this construction. 

Given a permutation $\pi=\pi_1\pi_2\ldots\pi_n\in S_n$, we obtain a standard diagram for $\pi$ by plotting the points of the form $(i,\pi_i)$ in the plane. A \emph{hook} in this diagram is the union of two line segments. One is a vertical line segment connecting a point $(i,\pi_i)$ to a point $(i,\pi_j)$, where $i<j$ and $\pi_i<\pi_j$. The second is a horizontal line segment connecting the points $(i,\pi_j)$ and $(j,\pi_j)$. One can think of drawing a hook by starting at the point $(i,\pi_i)$, moving upward, and then turning right to meet with the point $(j,\pi_j)$. The point $(i,\pi_i)$ is the \emph{southwest endpoint} of the hook while $(j,\pi_j)$ is the \emph{northeast endpoint} of the hook. We let ${}_eH$ and $H^e$ denote the southwest and northeast endpoints, respectively, of the hook $H$. 

In a valid hook configuration of $\pi$, each entry can be the southwest endpoint of at most one hook, but an entry can be the northeast endpoint of several hooks. If $i$ is a descent of $\pi$, then we require $(i,\pi_i)$ to be the southwest endpoint of a hook. On the other hand, if $(j,\pi_j)$ is a northeast endpoint of some hook, then at least one of the hooks with northeast endpoint $(j,\pi_j)$ must have a southwest endpoint $(k,\pi_k)$ for some descent $k$ of $\pi$. We also insist that if $(j,\pi_j)$ is a northeast endpoint of some hook, then there is some hook whose northeast endpoint is $(j,\pi_j)$ and whose southwest endpoint is $(j-1,\pi_{j-1})$. Finally, in order for a configuration of hooks to be valid, we impose a restriction on the projections of hooks onto the $x$-axis. Let $H$ and $H'$ be hooks, and let $I$ and $I'$ be their respective projections onto the $x$-axis (which are intervals). Let ${}_eH=(i,\pi_i)$, $H^e=(j,\pi_j)$, ${}_eH'=(i',\pi_{i'})$, and ${H'}^e=(j',\pi_{j'})$. If $I\cap I'$ contains more than one point and $\pi_j\leq\pi_{j'}$, then we require $I\subseteq I'$. Any configuration of hooks that satisfies these criteria is valid. Valid hook configurations are of fundamental importance in the development of our results, so we will make their definition formal. 

\begin{definition}\label{Def1}
Let $\pi\in S_n$. We say that an $m$-tuple $\mathscr H=(H_1,H_2,\ldots,H_m)$ is a \emph{valid hook configuration of} $\pi$ if $H_1,H_2,\ldots,H_m$ are hooks in the diagram of $\pi$ that satisfy the following properties. 
\begin{enumerate}[(a)]
\item If ${}_eH_\ell=(i_\ell,\pi_{i_\ell})$ for each $\ell\in[m]$, then $i_1<i_2<\cdots<i_m$. 
\item If $i$ is a descent of $\pi$, then $(i,\pi_i)={}_eH_\ell$ for some $\ell\in[m]$. 
\item If $(j,\pi_j)=H_\ell^e$ for some $\ell\in[m]$, then there exist $\ell',\ell''\in[m]$ such that the $x$-coordinate of ${}_eH_{\ell'}$ is a descent of $\pi$, $(j-1,\pi_{j-1})={}_eH_{\ell''}$, and $H_{\ell'}^e=H_{\ell''}^e=(j,\pi_j)$. 
\item If $\ell,\ell'\in [m]$, ${}_eH_\ell=(i,\pi_i)$, $H_\ell^e=(j,\pi_j)$, ${}_eH_{\ell'}=(i',\pi_{i'})$, $H_{\ell'}^e=(j',\pi_{j'})$, $\pi_j\leq\pi_{j'}$, and $\left|[i,j]\cap[i',j']\right|>1$, then $[i,j]\subseteq[i',j']$. 
\end{enumerate}
Let $SW(\mathscr H)=\{{}_eH_1,{}_eH_2,\ldots,{}_eH_m\}$ and $NE(\mathscr H)=\{H_1^e,H_2^e,\ldots,H_m^e\}$. Let $\mathcal H(\pi)$ denote the set of valid hook configurations of $\pi$.  
\end{definition}

Let us take a moment to briefly explain the necessity of some of the technical notions introduced in Definition \ref{Def1}. We are going to use the valid hook configurations of a permutation $\pi$ to construct trees whose postorders are $\pi$. In such a tree, the entries $\pi_j$ such that $(j,\pi_j)\in NE(\mathscr H)$ will be the parents of the descent tops of $\pi$. Furthermore, the hooks will become edges in the tree. More precisely, if $(j,\pi_j)\in NE(\mathscr H)$, then $\pi_i$ will be a child of $\pi_j$ if and only if $(i,\pi_i)$ is the southwest endpoint of some hook whose northeast endpoint is $(j,\pi_j)$. We require property (b) for this reason. This is also why property (c) guarantees that each element of $NE(\mathscr H)$ is the northeast endpoint of some hook whose southwest endpoint is $(i,\pi_i)$ for some descent $i$. Moreover, this explains why each point can be the southwest endpoint of at most one hook while it can be the northeast endpoint of multiple hooks (a node in a tree can have at most one parent, but it can have multiple children). Property (c) also states that if $(j,\pi_j)$ is a northeast endpoint of a hook, then it must be a northeast endpoint of some hook whose southwest endpoint is $(j-1,\pi_{j-1})$. Indeed, if $(j,\pi_j)$ is a northeast endpoint of some hook, then $\pi_j$ will not be a leaf in the tree we construct. Since $\pi$ is the postorder of this tree, this implies that the entry $\pi_{j-1}$ must be the rightmost child of $\pi_j$. Hence, $(j-1,\pi_{j-1})$ and $(j,\pi_j)$ must be connected by a hook.   

There are some immediate consequences of the criteria in the above definition that are useful to keep in mind. First, the only way that two hooks can intersect in exactly one point is if that point is the northeast endpoint of one of the hooks and the southwest endpoint of the other. Also, no entry in the diagram can lie above a hook. More formally, if $H_\ell$ is a hook with ${}_eH_\ell=(i,\pi_i)$ and $H_\ell^e=(j,\pi_j)$, then $\pi_k<\pi_j$ for all $k\in\{i+1,i+2,\ldots,j-1\}$. Indeed, suppose instead that $\pi_k>\pi_j$ for some $k\in\{i+1,i+2,\ldots,j-1\}$. Then $\pi$ must have a descent $i'\in\{k,k+1,\ldots,j-1\}$ such that $\pi_{i'}>\pi_j$. According to criterion (b) in the above definition, $(i',\pi_{i'})={}_eH_{\ell'}$ for some hook $H_{\ell'}$. Let $H_{\ell'}^e=(j',\pi_{j'})$. Since $\pi_j<\pi_{i'}<\pi_{j'}$ and $[i',i'+1]\subseteq[i,j]\cap [i',j']$, condition (d) in the above definition states that we must have $[i,j]\subseteq [i',j']$. However, this is impossible because $i<i'$. 

After drawing a diagram of a permutation $\pi$ with a valid hook configuration $\mathscr H=(H_1,H_2,\ldots,$ $H_m)\in\mathcal H(\pi)$, we can color the diagram with $m+1$ colors $c_0,c_1,\ldots,c_m$ as follows. First, if ${}_eH_\ell=(i,\pi_i)$ and $H_\ell^e=(j,\pi_j)$, then we refer to the line segment connecting the points $(i+1/2,\pi_j)$ and $(j,\pi_j)$ as the \emph{top part} of the hook $H_\ell$. Assign $H_\ell$ the color $c_\ell$ for each $\ell\in[m]$. Color each point $(k,\pi_k)$ as follows. Start at $(k,\pi_k)$, and move directly upward until hitting the top part of a hook. Color $(k,\pi_k)$ the same color as the hook that you hit. If you hit multiple hooks at once, use the color of the hook that was hit whose southwest endpoint if farthest to the right. If you do not hit the top part of any hook, give $(k,\pi_k)$ the color $c_0$. Note that if $(k,\pi_k)={}_eH_\ell$ for some $\ell\in [m]$, then we ignore the hook $H_\ell$ while moving upward from $(k,\pi_k)$ to find the hook that is to lend its color to $(k,\pi_k)$. Moreover, if $(k,\pi_k)\in NE(\mathscr H)$, then we give $(k,\pi_k)$ the color $c_r$, where $r$ is the largest element of $[m]$ such that $(k,\pi_k)=H_r^e$. 

\begin{example}\label{Exam2}
Figure \ref{Fig2} depicts the colored diagram obtained from a valid hook configuration of the permutation $\pi=2.7.3.5.9.10.11.4.8.1.6.12.13.14.15.16$.  

\begin{figure}[t]
\begin{center} 
\includegraphics[height=7.0cm]{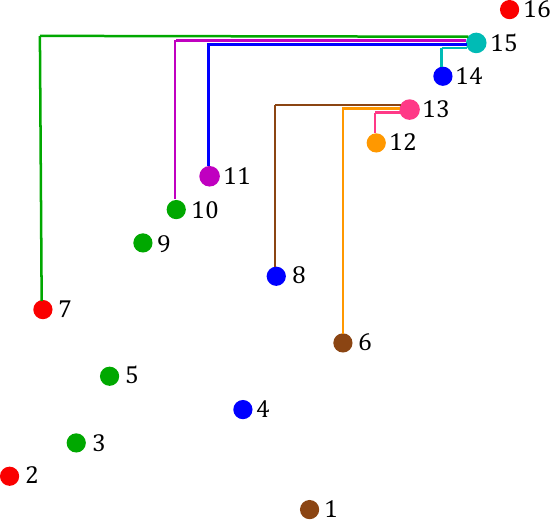}
\end{center}
\captionof{figure}{The colored diagram arising from a valid hook configuration.} \label{Fig2}
\end{figure}
\end{example}

\begin{remark}\label{Rem2}
We convene to say that the identity permutation $123\cdots n$ has a single empty valid hook configuration that induces a colored diagram in which every point is given the color $c_0$. 
\end{remark}

If $\pi\in S_n$, then each valid hook configuration $\mathscr H=(H_1,H_2,\ldots,H_m)\in\mathcal H(\pi)$ partitions $[n]$ into color classes. Let $Q_t(\mathscr H)$ be the set of entries $\pi_\ell$ such that  $(\ell,\pi_\ell)$ is assigned the color $c_t$. If we let $q_t(\mathscr H)=\vert Q_t(\mathscr H)\vert$, then $(q_0(\mathscr H),q_1(\mathscr H),\ldots,q_m(\mathscr H))$ is a composition of $n$ into $m+1$ parts. Let $\vert \mathscr H\vert=m$ denote the number of hooks in the valid hook configuration $\mathscr H$. In the next section, we show that the number of decreasing $\mathbb N$-trees whose postorder is $\pi$ is given by 
\begin{equation}\label{Eq1}
\sum_{\mathscr H\in\mathcal H(\pi)}\prod_{t=0}^{\vert \mathscr H\vert}C_{q_t(\mathscr H)-1},
\end{equation} where $C_i$ is the $i^\text{th}$ Catalan number. In Section 4, we refine this enumeration and provide analogous results aimed at answering Question \ref{Quest1} for other collections of tress $Y$. 

\begin{remark}\label{Rem1}
Suppose $\mathscr H\in\mathcal H(\pi)$ and $(j,\pi_j)\in NE(\mathscr H)$. If $(j,\pi_j)$ is given the color $c_t$ in the colored diagram of $\pi$ induced by $\mathscr H$, then $Q_t(\mathscr H)=\{\pi_j\}$ and $q_t(\mathscr H)=1$. Indeed, the hook colored $c_t$ is the hook with southwest endpoint $(j-1,\pi_{j-1})$ and northeast endpoint $(j,\pi_j)$. 
\end{remark}

Before we proceed to the next section, we record the following simple but useful lemma, which essentially states that the color classes mentioned above form increasing subsequences in $\pi$. 

\begin{lemma}\label{Lem1}
Let $\pi\in S_n$, and let $\mathscr H=(H_1,H_2,\ldots,H_m)\in\mathcal H(\pi)$. Let $r\in\{0,1,\ldots,m\}$. If $Q_r(\mathscr H)=\{(h_1,\pi_{h_1}),(h_2,\pi_{h_2}),\ldots,(h_k,\pi_{h_k})\}$, where $h_1<h_2<\cdots<h_k$, then $\pi_{h_1}<\pi_{h_2}<\cdots<\pi_{h_k}$. 
\end{lemma}
\begin{proof}
Suppose instead that $h_p<h_{p'}$ and $\pi_{h_p}>\pi_{h_{p'}}$ for some $p,p'\in[k]$. There is some descent $i$ of $\pi$ such that $h_p\leq i<h_{p'}$. Let us choose $i$ maximally. According to part (b) of Definition \ref{Def1}, $(i,\pi_i)={}_eH_\ell$ for some $\ell\in[m]$. Let $H_\ell^e=(j,\pi_j)$. Then $\pi_j>\pi_i$, so $j>h_{p'}$ because $i$ is the largest descent of $\pi$ that is less than $h_{p'}$. It follows that $(h_{p'},\pi_{h_{p'}})$ lies below the hook $H_\ell$ while $(h_p,\pi_{h_p})$ does not. Since $H_r$ is supposed to be the lowest hook lying above $(h_{p'},\pi_{h_{p'}})$, $H_r$ must lie below $H_\ell$. Since $(h_p,\pi_{h_p})$ lies below $H_r$ but does not lie below $H_\ell$, this means that $H_r$ lies below $H_\ell$ but does not lie \emph{completely} below $H_\ell$. In other words, we have contradicted part (d) of Definition \ref{Def1}. 
\end{proof}

\section{From Entries to $\mathbb N$-Trees}

Throughout this section, let $\pi=\pi_1\pi_2\cdots\pi_n\in S_n$. Let $\mathscr H=(H_1,H_2,\ldots,H_m)\in\mathcal H(\pi)$. At the end of the previous section, we defined $Q_t(\mathscr H)$ to be the set of entries $\pi_\ell$ such that $(\ell,\pi_\ell)$ is given the color $c_t$ in the colored diagram induced by $\mathscr H$. For each $t\in\{0,1,\ldots,m\}$, let $T_t$ be a decreasing $\mathbb N$-tree on $Q_t(\mathscr H)$ whose postorder lists the elements of $Q_t(\mathscr H)$ in increasing order. Let $\mathscr T=(T_0,T_1,\ldots,T_m)$. Choosing each tree $T_t$ amounts to choosing the (unlabeled) $\mathbb N$-tree with $q_t(\mathscr H)=\vert Q_t(\mathscr H)\vert$ vertices that serves as the underlying shape of $T_t$; the labeling of $T_t$ is then completely determined by the requirement that $P(T_t)$, the postorder of $T_t$, be in increasing order. Therefore, the number of ways to choose the tree $T_t$ is equal to the number of $\mathbb N$-trees with $q_t(\mathscr H)$ vertices, which is $C_{q_t(\mathscr H)-1}$. It follows that the number of ways to choose $\mathscr T$ is $\displaystyle{\prod_{t=0}^mC_{q_t(\mathscr H)-1}}$.  

\begin{example}\label{Exam3}
Figure \ref{Fig3} shows one possible collection of trees $T_0,T_1,\ldots,T_7$ that could arise from the permutation $\pi$ and the valid hook configuration $\mathscr H$ given in Example \ref{Exam2}. Observe that the postorder of each one of these trees is in increasing order. 

\begin{figure}[t]
\begin{center} 
\includegraphics[width=.6\linewidth]{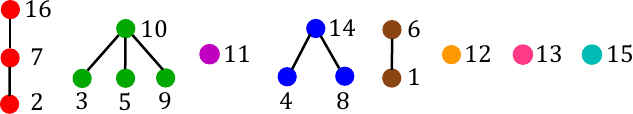}
\end{center}
\captionof{figure}{A collection of $\mathbb N$-trees $T_0,T_1,\ldots,T_7$.} \label{Fig3}
\end{figure}
\end{example}

We are going to describe a procedure for building $\mathbb N$-trees whose postorders are $\pi$ from the valid hook configuration
$\mathscr H$ and the collection of trees $T_0,T_1,\ldots,T_m$. We will then show that this procedure provides a one-to-one correspondence between pairs $(\mathscr H,\mathscr T)$ and decreasing $\mathbb N$-trees whose postorders are $\pi$. This is how we obtain the formula \eqref{Eq1} for the number of decreasing $\mathbb N$-trees with postorder $\pi$. 

The procedure works by constructing a sequence of decreasing $\mathbb N$-trees $\tau_n,\tau_{n-1},\ldots,\tau_1$, where the vertices of $\tau_i$ are the entries $\pi_i,\pi_{i+1},\ldots,\pi_n$. The final tree in the sequence, $\tau_1$, will be the $\mathbb N$-tree with postorder $\pi$ that we want. To begin the procedure, let $\tau_n$ be the tree consisting of the single vertex $\pi_n$. Now, suppose we have built the trees $\tau_n,\tau_{n-1},\ldots,\tau_{\ell+1}$. We will build the tree $\tau_\ell$ from $\tau_{\ell+1}$ by attaching $\pi_\ell$ as a leaf under one of the vertices of $\tau_{\ell+1}$. We consider two cases. 

Case 1: Suppose $(\ell,\pi_\ell)\in SW(\mathscr H)$. In this case, $(\ell,\pi_\ell)$ is the southwest endpoint of a unique hook $H_i$. Let $H_i^e=(j,\pi_j)$. To build the tree $\tau_\ell$ from $\tau_{\ell+1}$, attach $\pi_\ell$ as child of $\pi_j$ so that $\pi_\ell$ is the first (leftmost) child of $\pi_j$ in the tree $\tau_\ell$. 

Case 2: Suppose $(\ell,\pi_\ell)\not\in SW(\mathscr H)$. In this case, let $c_r$ be the color assigned to $(\ell+1,\pi_{\ell+1})$ (so $(\ell+1,\pi_{\ell+1})\in Q_r(\mathscr H)$). Let $u$ be the largest element of the set $[\ell]$ such that $(u,\pi_u)\in Q_r(\mathscr H)$. Let $\pi_v$ be the parent of $\pi_u$ in the tree $T_r$. To build the tree $\tau_\ell$ from $\tau_{\ell+1}$, attach $\pi_\ell$ as child of $\pi_v$ so that $\pi_\ell$ is the first (leftmost) child of $\pi_v$ in the tree $\tau_\ell$. 

In Case 2, we need to make sure that $u$ always exists and that the parent $\pi_v$ of $\pi_u$ in $T_r$ exists and is also a vertex in $\tau_{\ell+1}$. To show that $u$ exists, we need to show that there is some $i\in[\ell]$ such that $(i,\pi_i)\in Q_r(\mathscr H)$. If $r=0$, then we may set $i=1$, so assume $r>0$. Recall from part (c) of Definition \ref{Def1} that if $(j,\pi_j)\in NE(\mathscr H)$, then $(j-1,\pi_{j-1})\in SW(\mathscr H)$. Therefore, $(\ell+1,\pi_{\ell+1})\not\in NE(\mathscr H)$. Let $(h,\pi_h)={}_eH_r$. The points $(\ell,\pi_\ell)$ and $(\ell+1,\pi_{\ell+1})$ must both lie below $H_r$, so $h<\ell$. This means that $h+1\in[\ell]$ and $(h+1,\pi_{h+1})\in Q_r(\mathscr H)$, so we may set $i=h+1$. Now, Lemma \ref{Lem1} tells us that the entries $\pi_s$ such that $(s,\pi_s)\in Q_r(\mathscr H)$ form an increasing subsequence of $\pi$.  Consequently, $\pi_{\ell+1}>\pi_u$. It follows that $\pi_u$ cannot be the root of $T_r$, so $\pi_u$ has a parent $\pi_v$ in $T_r$. Because $T_r$ is a decreasing tree, $\pi_v>\pi_u$. Therefore, $v>u$ by Lemma \ref{Lem1}. Because $u$ was chosen to be the \emph{largest} element of $[\ell]$ such that $(u,\pi_u)\in Q_r(\mathscr H)$, it follows that $v\geq \ell+1$. Thus, $\pi_v$ is a vertex in $\tau_{\ell+1}$. 

\begin{figure}[t]
\begin{center} 
\includegraphics[height=5.0cm]{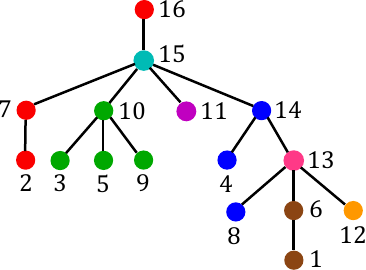}
\end{center}
\captionof{figure}{The tree $\tau_1$ obtained when $\pi$, $\mathscr H$, and $\mathscr T$ are as in Examples \ref{Exam2} and \ref{Exam3}.}\label{Fig4}
\end{figure}

Figure \ref{Fig4} shows the tree $\tau_1$ that results from this procedure when $\pi$, $\mathscr H$, and $\mathscr T$ are the permutation, the valid hook configuration, and the $8$-tuple of trees given in Example \ref{Exam2} and Example \ref{Exam3}. Observe that each of the trees $T_0,T_1,\ldots,T_7$ is, in some sense, embedded in this tree $\tau_1$. For example, the green tree $T_1$ is an actual subgraph of $\tau_1$. If we imagine deleting the vertex $13$ and then contracting the edge that joined the vertices $8$ and $13$, we will see a copy of the blue tree $T_3$ sitting inside of the larger tree. By appealing to the description of the above procedure, the reader may convince herself that this phenomenon will always occur for all of the trees $T_0,T_1,\ldots,T_m$. More precisely, if $\pi_i$ and $\pi_j$ are vertices of a tree $T_t$, then $\pi_i$ is in the $h^\text{th}$ subtree (meaning the $h^\text{th}$ subtree from the left) of $\pi_j$ in $\tau_1$ if and only if $\pi_i$ is in the $h^{\text{th}}$ subtree of $\pi_j$ in $T_t$. We make use of this fact in the proof of the following proposition. 

\begin{proposition}\label{Prop1}
Let $\tau_n,\tau_{n-1},\ldots,\tau_1$ be the sequence of decreasing $\mathbb N$-trees constructed in the above procedure. The postorder reading of the tree $\tau_\ell$ is $P(\tau_\ell)=\pi_\ell\pi_{\ell+1}\cdots\pi_n$. In particular, $P(\tau_1)=\pi$.  
\end{proposition}
\begin{proof}
We prove the proposition by inducting backward on $\ell$, noting first that the claim is trivial if $\ell=n$. Suppose $\ell<n$, and assume that $P(\tau_{\ell+1})=\pi_{\ell+1}\pi_{\ell+2}\cdots\pi_n$. Because $\tau_\ell$ is built by attaching $\pi_\ell$ to the tree $\tau_{\ell+1}$ as a leaf, the postorder of $\tau_\ell$ is obtained by inserting $\pi_\ell$ somewhere into $P(\tau_\ell)$. In other words, there are strings $L$ and $R$ such that $P(\tau_{\ell+1})=LR$ and $P(\tau_\ell)=L\pi_\ell R$. In order to complete the proof of the proposition, we need to show that $L$ is empty. This amounts to showing that $\pi_\ell$ precedes $\pi_{\ell+1}$ in $P(\tau_\ell)$.  

Suppose first that $(\ell,\pi_\ell)\in SW(\mathscr H)$. Let $H_s$ be the hook with ${}_eH_s=(\ell,\pi_\ell)$, and let $(j,\pi_j)=H_s^e$. Note that the point $(\ell+1,\pi_{\ell+1})$ is given the color $c_s$. It follows from the description of the procedure used to build the sequence of trees $\tau_n,\tau_{n-1},\ldots,\tau_{\ell+1}$ that $\pi_{\ell+1}$ is a descendant of $\pi_j$ in $\tau_{\ell+1}$ (therefore, also in $\tau_\ell$). In $\tau_\ell$, the vertex $\pi_\ell$ is a leaf which is the leftmost child of $\pi_j$. Consequently, $\pi_\ell$ precedes $\pi_{\ell+1}$ in $P(\tau_\ell)$. 

Next, suppose $(\ell,\pi_\ell)\not\in SW(\mathscr H)$. Preserve the notation from the description of the procedure in Case 2 above. Lemma \ref{Lem1} tells us that $\pi_u<\pi_{\ell+1}$ because $u<\ell+1$. We constructed $T_r$ so that its postorder would be in increasing order, so $\pi_u$ precedes $\pi_{\ell+1}$ in $P(T_r)$. In the paragraph preceding this proposition, we mentioned that $T_r$ is embedded in $\tau_1$. Therefore, $\pi_u$ precedes $\pi_{\ell+1}$ in $P(\tau_1)$. Moreover, $\pi_u$ must be a descendant of $\pi_v$ in $\tau_1$ because $\pi_u$ is a descendant of $\pi_v$ in $T_r$. Now, suppose that $\pi_{\ell+1}$ precedes $\pi_\ell$ in $P(\tau_\ell)$. Then $\pi_{\ell+1}$ precedes $\pi_\ell$ in $P(\tau_1)$, so $\pi_{\ell+1}$ appears between $\pi_u$ and $\pi_\ell$ in $P(\tau_1)$. Because $\pi_u$ and $\pi_\ell$ are both descendants of $\pi_v$ in $\tau_1$, $\pi_{\ell+1}$ must be a descendant of $\pi_v$ in $\tau_1$. Consequently, $\pi_{\ell+1}$ is a descendant of $\pi_v$ in $\tau_\ell$. However, this implies that $\pi_\ell$ precedes $\pi_{\ell+1}$ in $P(\tau_\ell)$ because $\pi_\ell$ is a leaf which is the leftmost child of $\pi_v$ in $\tau_\ell$. 
\end{proof} 

Now that we have shown that the above procedure creates a tree $\tau_1$ with postorder $\pi$, our goal is to show that each decreasing $\mathbb N$-tree with postorder $\pi$ is obtained in this way from a unique pair $(\mathscr H,\mathscr T)$. In the following theorem, let $g(\mathscr H,\mathscr T)$ denote the tree $\tau_1$ that our procedure produces from the valid hook configuration $\mathscr H$ and the tuple of decreasing $\mathbb N$-trees $\mathscr T$. It might be useful to keep Example \ref{Exam2}, Example \ref{Exam3}, and Figure \ref{Fig4} in mind during the following proof. 
\begin{theorem}\label{Thm1}
Let $\tau$ be a decreasing $\mathbb N$-tree whose postorder reading is $P(\tau)=\pi$. There is a unique pair $(\mathscr H,\mathscr T)$ such that $\tau=g(\mathscr H,\mathscr T)$.
\end{theorem}
\begin{proof}
Let $N$ be the set of entries $\pi_j$ of $\pi$ such that $\pi_j$ is a parent of a descent top of $\pi$ in $\tau$. Suppose $g(\mathscr H,\mathscr T)=\tau$, where $\mathscr H=(H_1,H_2,\ldots,H_m)$ and $\mathscr T=(T_0,T_1,\ldots,T_m)$. According to Definition \ref{Def1}, a point $(j,\pi_j)$ is a northeast endpoint of a hook in $\mathscr H$ if and only if it is the northeast endpoint of a hook whose southwest endpoint is $(i,\pi_i)$ for some descent $i$ of $\pi$. Referring to Case 1 of the procedure, this happens if and only if $\pi_j$ is the parent of $\pi_i$ in $\tau$. Therefore, $NE(\mathscr H)$ is precisely the collection of points
$(j,\pi_j)$ such that $\pi_j\in N$. We can now determine exactly what the hooks in $\mathscr H$ must be by observing that the points in $SW(\mathscr H)$ correspond to the children in $\tau$ of the points $\pi_j$ such that $(j,\pi_j)\in NE(\mathscr H)$. More precisely, there is a hook $H$ in $\mathscr H$ with ${}_eH=(h,\pi_h)$ and $H^e=(j,\pi_j)$ if and only if $\pi_j\in N$ and $\pi_h$ is a child of $\pi_j$ in $\tau$. Hence, $\mathscr H$ is uniquely determined by $\tau$. 

The paragraph preceding the proof of Proposition \ref{Prop1} tells us that each tree $T_t$ is embedded in $\tau$. More precisely, if $t\in\{0,1,\ldots,m\}$ and $(i,\pi_i),(j,\pi_j)\in Q_t(\mathscr H)$, then $\pi_i$ is in the $h^\text{th}$ subtree of $\pi_j$ in $T_t$ if and only if $\pi_i$ is in the $h^\text{th}$ subtree of $\pi_j$ in $\tau$. Therefore, $\mathscr T$ is uniquely determined by $\tau$. We have shown that there is at most one pair $(\mathscr H,\mathscr T)$ such that $g(\mathscr H,\mathscr T)=\tau$. 

On the other hand, it turns out that these properties of $\mathscr H$ and $\mathscr T$, which are necessary for $g(\mathscr H,\mathscr T)=\tau$, are sufficient. We may define a valid hook configuration $\mathscr H'=(H_1',H_2',\ldots,H_{m'}')$ by requiring that there is a hook in $\mathscr H$ with southwest endpoint $(i,\pi_i)$ and northeast endpoint $(j,\pi_j)$ if and only if $\pi_i$ is a child of $\pi_j$ in $\tau$ and $\pi_j\in N$. For each $t\in\{0,1,\ldots,m'\}$, define a decreasing $\mathbb N$-tree $T_t'$ on $Q_t(\mathscr H')$ by insisting that for any $(i,\pi_i),(j,\pi_j)\in Q_t(\mathscr H')$, $\pi_i$ is in the $h^\text{th}$ subtree of $\pi_j$ in $T_t'$ if and only if $\pi_i$ is in the $h^\text{th}$ subtree of $\pi_j$ in $\tau$. Let $\mathscr T'=(T_0',T_1',\ldots,T_{m'}')$. With these definitions, $g(\mathscr H',\mathscr T')=\tau$. 
\end{proof}
We are now able to obtain the formula \eqref{Eq1}. Recall the definition of $\mathcal H(\pi)$ from Definition \ref{Def1}. For any $\mathscr H\in\mathcal H(\pi)$, recall that $\vert \mathscr H\vert$ denotes the number of hooks in $\mathscr H$ and that $q_t(\mathscr H)$ denotes the number of points $(i,\pi_i)$ that are given the color $c_t$ in the colored diagram induced by $\mathscr H$. 
\begin{theorem}\label{Thm2}
Let $\pi\in S_n$. The number of decreasing $\mathbb N$-trees with postorder $\pi$ is \[\sum_{\mathscr H\in\mathcal H(\pi)}\prod_{t=0}^{\vert \mathscr H\vert}C_{q_t(\mathscr H)-1}.\]
\end{theorem}
\begin{proof}
Theorem \ref{Thm1} tells us that there is a one-to-one correspondence between decreasing $\mathbb N$-trees with postorder $\pi$ and pairs of the form $(\mathscr H,\mathscr T)$. In the beginning paragraph of this section, we showed that for any given $\mathscr H\in\mathcal H(\pi)$, the number of ways to choose $\mathscr T$ is $\displaystyle{\prod_{t=0}^{\vert \mathscr H\vert}C_{q_t(\mathscr H)-1}}$. 
\end{proof}
\section{Extensions and Refinements}
We continue to let $\pi=\pi_1\pi_2\cdots\pi_n$ be an arbitrary permutation in $S_n$. 

In the last section, we gave a procedure for constructing a sequence of decreasing $\mathbb N$-trees $\tau_n,\tau_{n-1},\ldots,\tau_1$ such that $P(\tau_1)=\pi$. The purpose of this section is to show how modifying that procedure can lead to results similar to Theorem \ref{Thm2} concerning different types of decreasing plane trees. 

Recall from the introduction that we defined two important types of plane trees: $S$-trees and $d$-ary plane trees. If $S$ is a set of nonnegative integers with $0\in S$, then an $S$-tree is constructed from a root along with a $j$-tuple of $S$-trees for some $j\in S$. If $d$ is a positive integer, then a $d$-ary plane tree is either empty or is a root along with a $d$-tuple of (possibly empty) $d$-ary plane trees. If $X$ is a set of positive integers, then a decreasing plane tree (of either type) on $X$ is a plane tree whose vertices have been labeled with the elements of $X$ so that the label of any nonroot vertex is smaller than the label of its parent. 

\begin{definition}\label{Def3}
If $\mathscr H\in\mathcal H(\pi)$ and $j\in[n]$, let $w_{j}(\mathscr H)$ denote the number of hooks in $\mathscr H$ with northeast endpoint $(j,\pi_j)$. 
If $0\in S\subseteq\mathbb N$, define $\mathcal H_S(\pi)$ to be the set of valid hook configurations $\mathscr H\in\mathcal H(\pi)$ such that $w_{j}(\mathscr H)\in S$ for all $j\in[n]$. 
\end{definition}

According to the definition of $NE(\mathscr H)$ in Definition \ref{Def1}, $(j,\pi_j)\in NE(\mathscr H)$ if and only if $w_{j}(\mathscr H)>0$. In the proof of Theorem \ref{Thm1}, we were given a decreasing $\mathbb N$-tree $\tau$ with $P(\tau)=\pi$ and showed that $\tau=g(\mathscr H,\mathscr T)$ for a unique pair $(\mathscr H,\mathscr T)$. We let $N$ be the set of entries $\pi_j$ such that $\pi_j$ was a parent of a descent top of $\pi$ in $\tau$. It then turned out that $\pi_j\in N$ if and only if $(j,\pi_j)\in NE(\mathscr H)$. The number of children of each vertex $\pi_j\in N$ was $w_{j}(\mathscr H)$, the number of hooks in $\mathscr H$ with northeast endpoint $(j,\pi_j)$. Furthermore, if $(j,\pi_j)\not\in NE(\mathscr H)$ and $\pi_j$ was a vertex in $T_t$, then the number of children of $\pi_j$ in $\tau$ was equal to the number of children of $\pi_j$ in $T_t$. We will make use of these observations in the proofs of the following theorems. 

\begin{theorem}\label{Thm3}
Let $S$ be a set of nonnegative integers with $0\in S$. Let $D_S(r)$ denote the number of $S$-trees with $r$ vertices. The number of decreasing $S$-trees with postorder $\pi$ is \[\sum_{\mathscr H\in\mathcal H_S(\pi)}\prod_{t=0}^{\vert \mathscr H\vert}D_S(q_t(\mathscr H)).\]
\end{theorem}
\begin{proof}
The proof is very similar to those of Theorems \ref{Thm1} and \ref{Thm2} because a decreasing $S$-tree is just a special type of decreasing $\mathbb N$-tree. As mentioned in the paragraph preceding this theorem, each tree $\tau$ with postorder $\pi$ corresponds to a pair $(\mathscr H,\mathscr T)$. If $(j,\pi_j)\in NE(\mathscr H)$, then the number of children of $\pi_j$ in $\tau$ is $w_{j}(\mathscr H)$. This is the reason for summing over the set $\mathcal H_S(\pi)$. Because the number of children of each of the other vertices of $\tau$ must be an element of $S$, the tuple of trees $\mathscr T=(T_0,T_1,\ldots,T_m)$ must be such that each tree $T_t$ is a decreasing $S$-tree whose postorder is in increasing order. This explains why we have replaced the product $\displaystyle{\prod_{t=0}^{\vert\mathscr H\vert}C_{q_t(\mathscr H)-1}}$ from Theorem \ref{Thm2} with the product $\displaystyle{\prod_{t=0}^{\vert\mathscr H\vert}D_S(q_t(\mathscr H))}$.   
\end{proof}

We can also modify our techniques in order to count decreasing $d$-ary trees. 

\begin{theorem}\label{Thm4}
Let $d$ be a positive integer. Let $E_d(r)$ denote the number of $d$-ary plane trees with $r$ vertices. The number of decreasing $d$-ary plane trees with postorder $\pi$ is \[\sum_{\mathscr H\in\mathcal H_{[d]\cup\{0\}}(\pi)}\left(\prod_{j=1}^n{d\choose w_{j}(\mathscr H)}\right)\left(\prod_{t=0}^{\vert \mathscr H\vert}E_d(q_t(\mathscr H))\right).\]
\end{theorem}
\begin{proof}
We have made two simple modifications to the ideas used in the proof of Theorem \ref{Thm3}. The first is that we have replaced $D_S(q_t(\mathscr H))$ with $E_d(q_t(\mathscr H))$ because we are considering different types of trees. The second is the introduction of the factors of the form $\displaystyle{\prod_{j=1}^n{d\choose w_{j}(\mathscr H)}}$. The reason for these new factors comes from the fact that vertices in decreasing $d$-ary trees can have empty subtrees. For each $(j,\pi_j)\in NE(\mathscr H)$, there are $\displaystyle{{d\choose w_{j}(\mathscr H)}}$ ways to choose which of the $d$ subtrees of $\pi_j$ are nonempty. When $(j,\pi_j)\not\in NE(\mathscr H)$, the factor $\displaystyle{{d\choose w_{j}(\mathscr H)}}$ does not change anything because $w_{j}(\mathscr H)=0$.   
\end{proof}

The following theorems serve to refine the enumerative results we have found so far. We omit their proofs because they are straightforward consequences of ideas that we have already used, especially the observations mentioned in the paragraph immediately preceding Theorem \ref{Thm3}. We need some notation before stating these results. If $R$ is a set of nonnegative integers, let $\Phi_R(u)$ denote the set of all decreasing plane trees (of either of our two types) in which there are exactly $u$ vertices $v$ such that the number of children of $v$ is an element of $R$. If $\mathscr H$ is a valid hook configuration of $\pi$, let $\Theta(\mathscr H)$ be the set of all $i\in\{0,1,\ldots,\vert\mathscr H\vert\}$ such that the color $c_i$ is not used in the colored diagram induced by $\mathscr H$ to color a point in $NE(\mathscr H)$. We let $\widehat{\vert\mathscr H\vert}=\vert\Theta(\mathscr H)\vert-1=\vert \mathscr H\vert-\vert NE(\mathscr H)\vert$. For example, the colors $c_6$ and $c_7$ (pink and teal) are used in Example \ref{Exam2} to colors the points in $NE(\mathscr H)$. Therefore, in that example, we have $\Theta(\mathscr H)=\{0,1,2,3,4,5\}$ and $\widehat{\vert\mathscr H\vert}=5$. 
If $\Theta(\mathscr H)=\left\{i_0,i_1,\ldots,i_{\widehat{\vert\mathscr H\vert}}\right\}$, where $i_0<i_1<\cdots<i_{\widehat{\vert\mathscr H\vert}}$, then we let $\widehat q_t(\mathscr H)=q_{i_t}(\mathscr H)$. In other words, the tuple $\left(\widehat q_0(\mathscr H),\widehat q_1(\mathscr H),\ldots,\widehat q_{\widehat{\vert \mathscr H\vert}}(\mathscr H)\right)$ is obtained by starting with the tuple $\left(q_0(\mathscr H),q_1(\mathscr H),\ldots,q_{\vert\mathscr H\vert}(\mathscr H)\right)$
and removing all of the coordinates $q_i(\mathscr H)$ such that the color $c_i$ is used to color a point in $NE(\mathscr H)$. 

\begin{theorem}\label{Thm5}
Let $R\subseteq S\subseteq\mathbb N$ with $0\in S$. Let $\mathcal H_S(\pi;R,u)$ be the set of $\mathscr H\in\mathcal H_S(\pi)$ such that $\vert\{(j,\pi_j)\in NE(\mathscr H)\colon w_{j}(\mathscr H)\in R\}\vert=u$. Let $D_S(r;R,u)$ be the number of $S$-trees in $\Phi_R(u)$ with $r$ vertices. The number of decreasing $S$-trees $\tau\in\Phi_R(p)$ with $P(\tau)=\pi$ is 
\[\sum_{u=0}^p\sum_{\mathscr H\in\mathcal H_S(\pi;R,u)}\,\sum_{j_0+j_1+\cdots+j_{\widehat{\vert\mathscr H\vert}}=p-u}\,\prod_{t=0}^{\widehat{\vert\mathscr H\vert}}D_S(\widehat q_t(\mathscr H);R,j_t),\] where the innermost sum ranges over all nonnegative integers $j_0,j_1,\ldots,j_{\widehat{\vert\mathscr H\vert}}$ that sum to $p-u$.   
\end{theorem}

\begin{theorem}\label{Thm6}
Let $R\subseteq [d]\cup\{0\}$. Let $\mathcal H_{[d]\cup\{0\}}(\pi;R,u)$ be the set of $\mathscr H\in\mathcal H_{[d]\cup\{0\}}(\pi)$ such that $\vert\{(j,\pi_j)\in NE(\mathscr H)\colon w_{j}(\mathscr H)\in R\}\vert=u$. Let $E_d(r;R,u)$ be the number of $d$-ary plane trees in $\Phi_R(u)$ with $r$ vertices. The number of decreasing $d$-ary plane trees $\tau\in\Phi_R(p)$ with $P(\tau)=\pi$ is 
\[\sum_{u=0}^p\sum_{\mathscr H\in\mathcal H_{[d]\cup\{0\}}(\pi;R,u)}\left(\prod_{j=1}^n{d\choose w_{j}(\mathscr H)}\right)\sum_{j_0+j_1+\cdots+j_{\widehat{\vert\mathscr H\vert}}=p-u}\prod_{t=0}^{\widehat{\vert\mathscr H\vert}}E_d(\widehat q_t(\mathscr H);R,j_t),\] where the innermost sum ranges over all nonnegative integers $j_0,j_1,\ldots,j_{\widehat{\vert\mathscr H\vert}}$ which sum to $p-u$.   
\end{theorem} 

It may seem as though Theorems \ref{Thm5} and \ref{Thm6} provide unnecessarily lengthy expressions by merely summing over appropriate sets. As these theorems stand, they appear too complicated to be of much use. However, our real interest lies in some special cases in which the given expressions simplify considerably. 

Recall from Definition \ref{Def1} that if $\mathscr H\in\mathcal H(\pi)$ and $(j,\pi_j)\in NE(\mathscr H)$, then $(j,\pi_j)$ is the northeast endpoint of at least two hooks: one with southwest endpoint $(j-1,\pi_{j-1})$ and one with southwest endpoint $(i,\pi_i)$ for some descent $i$ of $\pi$. It is natural to consider the set of valid hook configurations $\mathscr H$ in which each point in $NE(\mathscr H)$ is a northeast endpoint of \emph{exactly} two hooks. In the notation of Definition \ref{Def3}, this is the set $\mathcal H_{\{0,2\}}(\pi)$. It is particularly easy to work with this set of valid hook configurations because the number of hooks in any such configuration must be twice the number of descents of $\pi$.  

The following corollary invokes the numbers $D_{\{0,1,2\}}(r;R,u)$ for certain sets $R\subseteq\{0,1,2\}$. These numbers are known\footnote{See sequences A000108, A001006, A055151, A097610, A091894, A121448 in the Online Encyclopedia of Integer Sequences \cite{OEIS}.} and are listed in Table \ref{Tab1}. 

\begin{table}[h]

\begin{tabu}{|c|[1.5pt]c|c|}

    \hline $R$ & $D_{\{0,1,2\}}(r;R,u)$ & $E_2(r;R,u)$ \\\tabucline[2pt]{-}
    $\{0\}$ & $\frac{1}{u}{r-1\choose u-1}{r-u\choose u-1}$ & $2^{r-2u+1}{r-1\choose 2u-2}C_{u-1}$ \\\hline
    $\{1\}$ & ${r-1\choose u}C_{(r-u-1)/2}$ & $2^u{r-1\choose u}C_{(r-u-1)/2}$ \\\hline
    $\{2\}$ & $\frac{1}{u+1}{r-1\choose u}{r-u-1\choose u}$ & $2^{r-2u-1}{r-1\choose 2u}C_{u}$ \\\hline
    $\{0,1,2\}$ & $M_{r-1}\delta_{r,u}$ & $C_r\delta_{r,u}$ \\\hline
\end{tabu}\vspace{.3cm}
\caption{Values of $D_{\{0,1,2\}}(r;R,u)$ and $E_2(r;R,u)$ for $R\subseteq\{0,1,2\}$. Here, $M_{r-1}$ denotes the $(r-1)^{\text{th}}$ Motzkin number. We convene to let $C_x=0$ if $x\not\in\mathbb Z$. In the last row, $\delta_{r,u}$ is the Kronecker delta. Observe that if $R=\{0,1,2\}\setminus\{a\}$ for some $a\in\{0,1,2\}$, then $D_{\{0,1,2\}}(r;R,u)=D_{\{0,1,2\}}(r;\{a\},r-u)$ and $E_2(r;R,u)=E_2(r;\{a\},r-u)$.}\label{Tab1}
\end{table}

\begin{corollary}\label{Cor1}
Let $\pi\in S_n$ be a permutation with exactly $k$ descents. Let $R\subseteq \{0,1,2\}$. Let $\chi=k$ if $2\in R$, and let $\chi=0$ if $2\not\in R$. The number of decreasing $\{0,1,2\}$-trees $\tau\in\Phi_R(p)$ with $P(\tau)=\pi$ is \[\sum_{\substack{\mathscr H\in\mathcal H_{\{0,2\}}(\pi)\\j_0+j_1+\cdots+j_k=p-\chi}}\prod_{t=0}^k D_{\{0,1,2\}}(\widehat q_t(\mathscr H);R,j_t),\]
where the numbers $j_0,j_1,\ldots,j_k$ are assumed to be nonnegative integers. 
\end{corollary} 
\begin{proof}
Set $S=\{0,1,2\}$ in Theorem \ref{Thm5}. Note that $\mathcal H_{\{0,1,2\}}(\pi)=\mathcal H_{\{0,2\}}(\pi)$ because a point cannot be the northeast endpoint of exactly one hook. Suppose $\mathscr H\in\mathcal H_{\{0,2\}}(\pi)$. Each point in $NE(\mathscr H)$ is the northeast endpoint of two hooks, exactly one of which has a southwest endpoint $(i,\pi_i)$ for some descent $i$ of $\pi$. Therefore, $\vert NE(\mathscr H)\vert=k$ and $\vert\mathscr H\vert=2k$. By definition, $\widehat{\vert\mathscr H\vert}=\vert\mathscr H\vert-\vert NE(\mathscr H)\vert=k$. Because $w_{j}(\mathscr H)=2$ whenever $(j,\pi_j)\in NE(\mathscr H)$, $\mathcal H_{\{0,1,2\}}(\pi;R,u)$ is empty if $u\neq 0$ and $2\not\in R$. Similarly, $\mathcal H_{\{0,1,2\}}(\pi;R,u)$ is empty when $u\neq k$ and $2\in R$. When $u=0$ and $2\not\in R$, $\mathcal H_{\{0,1,2\}}(\pi;R,u)=\mathcal H_{\{0,2\}}(\pi)$. When $u=k$ and $2\in R$, $\mathcal H_{\{0,1,2\}}(\pi;R,u)=\mathcal H_{\{0,2\}}(\pi)$.  
\end{proof}
\begin{example}\label{Exam4}
By definition, $\Phi_{\{0\}}(\ell)$ is the set of all decreasing plane trees with $\ell$ leaves. According to Table \ref{Tab1}, $D_{\{0,1,2\}}(r;\{0\},u)=\frac{1}{u}{r-1\choose u-1}{r-u\choose u-1}$ if $u>0$. Therefore, Corollary \ref{Cor1} tells us that if $\pi\in S_n$ has exactly $k$ descents, then the number of decreasing $\{0,1,2\}$-trees with exactly $p$ leaves that have postorder $\pi$ is \[\sum_{\substack{\mathscr H\in\mathcal H_{\{0,2\}}(\pi)\\j_0+j_1+\cdots+j_k=p}}\prod_{t=0}^k \frac{1}{j_t}{\widehat q_t(\mathscr H)-1\choose j_t-1}{\widehat q_t(\mathscr H)-j_t\choose j_t-1}.\] It is understood that the numbers $j_0,j_1,\ldots,j_k$ in the sum are positive. 
\end{example}

In the next corollary, recall the definition of the numbers $E_d(r;R,u)$ from Theorem \ref{Thm6}. When $d=2$ and $R\subseteq\{0,1,2\}$, these numbers are given in Table \ref{Tab1}. 

\begin{corollary}\label{Cor2}
Let $\pi\in S_n$ be a permutation with exactly $k$ descents. Let $R\subseteq \{0,1,2\}$. Let $\chi=k$ if $2\in R$, and let $\chi=0$ if $2\not\in R$. The number of decreasing binary plane trees $\tau\in\Phi_R(p)$ with $P(\tau)=\pi$ is \[\sum_{\substack{\mathscr H\in\mathcal H_{\{0,2\}}(\pi)\\j_0+j_1+\cdots+j_k=p-\chi}}\prod_{t=0}^k E_2(\widehat q_t(\mathscr H);R,j_t),\]
where the numbers $j_0,j_1,\ldots,j_k$ are assumed to be nonnegative integers. 
\end{corollary} 
\begin{proof}
Set $d=2$ in Theorem \ref{Thm5}. The simplification of the formula in Theorem \ref{Thm6} results from the same argument used to simplify the formula from Theorem \ref{Thm5} in the proof of Corollary \ref{Cor1}. In this case, we must also consider the expression $\displaystyle{\prod_{j=1}^n{d\choose w_{j}(\mathscr H)}}$. This product simplifies to $1$ because $d=2$ and $w_{j}\in\{0,2\}$ for all $j\in[n]$ and $\mathscr H\in\mathcal H_{\{0,2\}}(\pi)$. 
\end{proof}
\begin{example}\label{Exam5}
According to Table \ref{Tab1}, $E_2(r;\{0\},u)=2^{r-2u+1}{r-1\choose 2u-2}C_{u-1}$ if $u>0$. Therefore, Corollary \ref{Cor1} tells us that if $\pi\in S_n$ has exactly $k$ descents, then the number of decreasing  binary plane trees with exactly $p$ leaves that have postorder $\pi$ is \[\sum_{\substack{\mathscr H\in\mathcal H_{\{0,2\}}(\pi)\\j_0+j_1+\cdots+j_k=p}}\prod_{t=0}^k 2^{\widehat q_t(\mathscr H)-2j_t+1}{\widehat q_t(\mathscr H)-1\choose 2j_t-2}C_{j_t-1}.\] Because $\widehat q_0(\mathscr H)+\widehat q_1(\mathscr H)+\cdots+\widehat q_k(\mathscr H)=n-k$, this expression simplifies further to  \[2^{n-2p+1}\sum_{\substack{\mathscr H\in\mathcal H_{\{0,2\}}(\pi)\\j_0+j_1+\cdots+j_k=p}}\prod_{t=0}^k {\widehat q_t(\mathscr H)-1\choose 2j_t-2}C_{j_t-1}.\] 
\end{example}

\section{The Stack-Sorting Algorithm}

In his 1990 Ph.D. thesis, Julian West \cite{West90} studied a function $s$ that transforms permutations into permutations through the use of a vertical stack. We call the function $s$ the \emph{deterministic stack-sorting algorithm}. Given an input permutation $\sigma=\sigma_1\sigma_2\cdots\sigma_n$, the permutation $s(\sigma)$ is computed as follows. At any point in time during the algorithm, if the leftmost entry in the input permutation is larger than the entry at the top of the stack or if the stack is empty, the leftmost entry in the input permutation is placed at the top of the stack. Otherwise, the entry at the top of the stack is annexed to the right end of the growing output permutation. For example, $s(35214)=31245$. 

Example \ref{Exam5} has a natural interpretation in terms of the deterministic stack-sorting algorithm. This is because the map $s$ is intimately related to the postorder readings of decreasing binary plane trees. If we are given any decreasing binary plane tree, we can read its labels in \emph{symmetric order} by first reading the left subtree of the root in symmetric order, then reading the root, and finally reading the right subtree of the root in symmetric order. If we let $\mathcal D_n$ denote the set of decreasing binary plane trees on $[n]$, then the map $S\colon\mathcal D_n\to S_n$ that sends a tree to its symmetric order reading is a bijection \cite{Stanley}. Furthermore, one can show that $s(\sigma)=P(S^{-1}(\sigma))$ for any $\sigma\in S_n$ \cite[Corollary 8.22]{Bona}. Therefore, the preimages of a permutation $\pi\in S_n$ under $s$ are in bijective correspondence with the decreasing binary plane trees with postorder $\pi$ (the bijection being the map $S$). 

A valley of a permutation $\sigma_1\sigma_2\cdots\sigma_n$ is an index $i\in\{1,2,\ldots,n\}$ such that $\sigma_i<\min\{\sigma_{i-1},\sigma_{i+1}\}$, where we make the convention $\sigma_0=\sigma_{n+1}=\infty$. A leaf of a decreasing binary plane tree corresponds to a valley in the tree's symmetric order reading. Therefore, we may rephrase the result from Example \ref{Exam5} as follows: 

\begin{corollary}\label{Cor3}
If $\pi\in S_n$ is a permutation with exactly $k$ descents, then the number of permutations $\sigma\in S_n$ that have exactly $p$ valleys and that satisfy $s(\sigma)=\pi$ is given by
\[2^{n-2p+1}\sum_{\substack{\mathscr H\in\mathcal H_{\{0,2\}}(\pi)\\j_0+j_1+\cdots+j_k=p}}\prod_{t=0}^k {\widehat q_t(\mathscr H)-1\choose 2j_t-2}C_{j_t-1}.\] 
\end{corollary}    

Corollary \ref{Cor3} provides a method for calculating a refined enumeration of the preimages of any permutation $\pi$ under $s$. The reader may very well ask ``what about the \emph{total} number of preimages of $\pi$ under $s$?" West originally asked about preimages of permutations under $s$, and he defined the \emph{fertility} of a permutation to be the total number of these preimages \cite{West90}. Subsequently, Bousquet-M\'elou gave a method for determining whether or not the feritility of any permutation is $0$ (that is, whether or not a permutation is in the image of $s$), and she stated that it would be interesting to find a method for computing the fertility of any given permutation \cite{Bousquet00}. Fortunately, we have such a method (reliant on the construction of the set $\mathcal H_{\{0,2\}}(\pi)$). 

\begin{theorem}\label{Thm7}
If $\pi\in S_n$ is a permutation with exactly $k$ descents, then the number of permutations $\sigma\in S_n$ such that $s(\sigma)=\pi$ is given by \[\sum_{\mathscr H\in\mathcal H_{\{0,2\}}(\pi)}\prod_{t=0}^kC_{\widehat{q}_t(\mathscr H)}.\]
\end{theorem} 
\begin{proof}
Set $d=2$ in Theorem \ref{Thm4}, and use the fact that $\mathcal H_{[2]\cup\{0\}}(\pi)=\mathcal H_{\{0,2\}}(\pi)$. By definition, $w_{j}(\mathscr H)\in \{0,2\}$ for all $j\in [n]$ and $\mathscr H\in\mathcal H_{\{0,2\}}(\pi)$. Therefore, the first product in the summation in Theorem \ref{Thm4} simplifies to $1$. The second product becomes $\displaystyle{\prod_{t=0}^{\vert\mathscr H\vert}C_{q_t(\mathscr H)}}$. This is because $E_2(q_t(\mathscr H))$, the number of (unlabeled) binary plane trees with $q_t(\mathscr H)$ vertices, is the Catalan number $C_{q_t(\mathscr H)}$. Finally, note that if $c_t$ is a color used to color an element of $NE(\mathscr H)$ in the colored diagram of $\pi$ induced by $\mathscr H$, then $C_{q_t(\mathscr H)}=C_1=1$ by Remark \ref{Rem1}. Therefore, the product simplifies further to $\displaystyle{\prod_{t=0}^kC_{\widehat{q}_t(\mathscr H)}}$.  
\end{proof}

As a final application, we count the number of preimages of a permutation under $s$ which have a fixed number of descents. In the following theorem, recall that the Narayana number $N(a,b)$ is defined by $N(a,b)=\frac{1}{a}{a\choose b}{a\choose b-1}$.  

\begin{theorem}\label{Thm8}
If $\pi\in S_n$ has exactly $k$ descents and $m$ is a nonnegative integer, then the number of permutations $\sigma\in S_n$ which have exactly $m$ descents and satisfy $s(\sigma)=\pi$ is given by \[\sum_{\substack{\mathscr H\in\mathcal H_{\{0,2\}}(\pi)\\ j_0+j_1+\cdots+j_k=m-k}}\prod_{t=0}^kN(\widehat{q}_t(\mathscr H),j_t+1),\] where the numbers $j_0,j_1,\ldots,j_k$ in the sum are nonnegative integers that sum to $m-k$. 
\end{theorem}
\begin{proof}
The descents of a permutation $\sigma$ are in one-to-one correspondence with the right edges of the decreasing binary plane tree $S^{-1}(\sigma)$. Therefore, we are really counting the number of decreasing binary plane trees that have exactly $m$ right edges and that have postorder $\pi$. 

Choose some $\mathscr H\in\mathcal H_{\{0,2\}}(\pi)$. Observe that if $(j,\pi_j)\in NE(\mathscr H)$, then $w_{j}(\mathscr H)=2$. This means that for any collection of trees $\mathscr T=(T_0,T_1,\ldots,T_{2k})$, $\pi_j$ will have two children in $g(\mathscr H,\mathscr T)$. In particular, $\pi_j$ will have a right child in $g(\mathscr H,\mathscr T)$. Consequently, $g(\mathscr H,\mathscr T)$ will automatically have $k$ right edges that attach the $k$ elements of the set $\{\pi_j\colon (j,\pi_j)\in NE(\mathscr H)\}$ to their right children. 

We now need to choose a collection of trees $\mathscr T=(T_0,T_1,\ldots,T_{2k})$ such that the total number of right edges in all of the trees $T_0,T_1,\ldots,T_{2k}$ is $m-k$. If $c_t$ is a color used to color a point in $NE(\mathscr H)$ in the colored diagram of $\pi$ induced by $\mathscr H$, then the number of vertices of $T_t$ is $q_t(\mathscr H)=1$ by Remark \ref{Rem1}. Such a tree has no right edges, and there is only one way to choose each such tree. Therefore, we are left to choose the trees $T_i$ for $i\in\Theta(\mathscr H)$. Let $\Theta(\mathscr H)=\{i_0,i_1,\ldots,i_k\}$, where $i_0<i_1<\ldots<i_k$. Choose a collection of nonnegative integers $j_0,j_1,\ldots,j_k$ with $j_0+j_1+\cdots+j_k=m-k$. It is well known that $N(a,b+1)$ is the number of unlabeled binary plane trees with $a$ vertices that have exactly $b$ right edges. Therefore, the number of ways to choose the trees $T_{i_0},T_{i_1},\ldots,T_{i_k}$ so that each tree $T_{i_t}$ has exactly $j_t$ right edges is $\displaystyle{\prod_{t=0}^kN(\widehat{q}_t(\mathscr H),j_t+1)}$. The result now follows by summing over all possible $j_0,j_1,\ldots,j_k$ and all possible $\mathscr H\in\mathcal H_{\{0,2\}}(\pi)$. 
\end{proof}

\section{Concluding Remarks}
The formulas given in this paper include sums over sets of the form $\mathcal H_S(\pi)$, where $0\in S\subseteq\mathbb N$ and $\pi\in S_n$. For this reason, it would be interesting to have a method for efficiently generating all of the valid hook configurations in the set $\mathcal H_S(\pi)$ if we are given $S$ and $\pi$. Alternatively, suppose we fix $S$ and let $\Lambda$ be a specific family of permutations (such as the family of layered permutations or the family of involutions). We would be interested in calculating (or at least estimating) the number of hooks in the set $\mathcal H_S(\pi)$ for each $\pi\in\Lambda$.

\section{Acknowledgments}
The author would like to thank Mikl\'os B\'ona for very helpful advise concerning the submission and organization of this paper. The author would also like to thank the anonymous referees for their useful suggestions.


\begin{thebibliography}{9}
\bibitem{Albert14}
M. Albert and M. Bouvel, Operators of equivalent sorting power and related Wilf-equivalences. DMTCS proc. AS, (2013), 701--712.

\bibitem{Bona}
B\'ona, Mikl\'os, Combinatorics of permutations. Second Edition. CRC Press, 2012. 

\bibitem{Bona02}
B\'ona, Mikl\'os, Symmetry and unimodality in $t$-stack sortable permutations. Journal of Combinatorial Theory, Series A 98.1 (2002): 201--209.

\bibitem{Bona03}
B\'ona, Mikl\'os, A survey of stack-sorting disciplines, Electron. J. Combin. 9.2 (2003): 16. 

\bibitem{Bousquet00}
Bousquet-Melou, Mireille, Sorted and/or sortable permutations, Discrete Math., 225
(2000), no. 1-3, 25–-50.

\bibitem{OEIS} 
\emph{The On-Line Encyclopedia of Integer Sequences}, published electronically at http://oeis.org, 2010.


\bibitem{Stanley}
Stanley, Richard, \emph{Enumerative Combinatorics, Volume 1,} Second Edition. Cambridge University Press, Cambridge UK, 2012.
   

\bibitem{West90} 
West, Julian, Permutations with restricted subsequences and stack-sortable permutations, Ph.D. Thesis, MIT,
1990. 

\end{thebibliography}
\end{document}